\documentclass[12pt]{amsart}
\usepackage{pdfsync}
\usepackage{graphicx,color}
\usepackage{amssymb}
\usepackage{amsmath}
\usepackage{amsthm}
\usepackage{enumerate}

\usepackage{esvect}

\usepackage{ucs}
\usepackage{cite}
\usepackage{amsxtra}
\usepackage[utf8x]{inputenc}
\usepackage{verbatim} 

\usepackage[left=3.5cm,top=2.5cm,right=3.5cm,bottom=2.5cm]{geometry}

\newcommand{\dist}{\text{dist}}

\newtheorem{theorem}{Theorem}
\newtheorem{definition}[theorem]{Definition}
\newtheorem{proposition}[theorem]{Proposition}
\newtheorem{corollary}[theorem]{Corollary}

\newtheorem{lemma}[theorem]{Lemma}

\newtheorem{mainthm}{Theorem}

\newtheorem{remark}[theorem]{Remark}

\newcommand{\R}{\mathbb{R}}
\def\Z{{\mathbb Z}}
\def\N{{\mathbb N}}

\def\C{{\mathbb C}}

\makeatletter
\def\moverlay{\mathpalette\mov@rlay}
\def\mov@rlay#1#2{\leavevmode\vtop{%
   \baselineskip\z@skip \lineskiplimit-\maxdimen
   \ialign{\hfil$\m@th#1##$\hfil\cr#2\crcr}}}
\newcommand{\charfusion}[3][\mathord]{
    #1{\ifx#1\mathop\vphantom{#2}\fi
        \mathpalette\mov@rlay{#2\cr#3}
      }
    \ifx#1\mathop\expandafter\displaylimits\fi}
\makeatother

\definecolor{Azul}{rgb}{0.0, 0.0, 1.0}
\definecolor{Rojo}{rgb}{1.0, 0.03, 0.0}
\definecolor{Purpura}{rgb}{0.5,0,0.35}
\definecolor{BurntOrange}{cmyk}{0,0.51,1,0}
\definecolor{PineGreen}{cmyk}{0.92,0,0.59,0.250}
\definecolor{Violeta}{rgb}{0.39,0.17,0.63}
\definecolor{Violeta2}{rgb}{0.6, 0.4, 0.8}
\definecolor{Fucsia}{rgb}{1.0, 0.01, 0.24}
\definecolor{Rosa}{rgb}{0.63,0.17,0.39}
\definecolor{VerdeManzana}{rgb}{0.39,0.63,0.17}
\definecolor{Celeste}{rgb}{0.7,0.7,1}

\begin{document}
\thispagestyle{empty}

\title{Large sets avoiding linear patterns.}

\author{Alexia Yavicoli}

\address{Department of mathematics and IMAS/CONICET\\
FCEN University of Buenos Aires.}

\email{ayavicoli@dm.uba.ar}

\thanks{A.Y.  was supported by projects UBACyT 2014-2017 20020130100403BA and PIP 11220110101018 (CONICET)}

\begin{abstract}
 We prove that for any dimension function $h$  with $h \prec x^d$ and for any countable set of linear patterns, there exists a compact set $E$ with $\mathcal{H}^h(E)>0$ avoiding all the given patterns. We also give several applications and recover results of Keleti, Maga, and Máthé.
\end{abstract}

\begin{center}
\vspace*{0.5cm}\small\em\flushright{Dedicated to Ursula Molter on the occasion of her 60 years.}\vspace*{1cm}
\end{center}

\maketitle


\section{Introduction.}

The classic Roth's Theorem  \cite{Roth} says that given $\delta>0$, there exists $N_0$ such that if $N \geq N_0$, then any subset of $\{1, \cdots ,N\}$ with at least $\delta N$ elements contains an arithmetic progression of length  $3$.

A major problem since then has been to find functions $h(N)$, as small as possible, such that for large enough $N$, a subset of  $\{1, \cdots , N\}$ with $h(N)$ elements contains an arithmetic progression of length $3$. The best known $h$ with this property is  $$h(N)=\frac{(\log\log N)^4 N}{\log N},$$ as shown by Bloom \cite{Bloom16}, slightly improving a previous remarkable result of Sanders \cite{Sanders11}.

In the opposite direction, Behrend \cite{Beh} showed that if $$h(N):=cNe^{-C\sqrt{\log(N)}},$$
where $c,C>0$ are absolute constants, then for all $N$ there exists a subset of $\{1,\cdots , N\}$ with at most $h(N)$ elements without arithmetic progressions. Note that, in particular, for all $\varepsilon>0$, we have $h(N)>N^{1-\varepsilon}$ if $N$ is large enough.

One of the motivations of this work is to investigate whether there can exist examples similar to Behrend's in the continuous case. In order to formulate the problem, we recall the definition of Hausdorff measure associated to an arbitrary dimension function.

We will use the notation $\vert U\vert$ for the diameter of the set $U$.
\begin{definition}
If $E\subseteq \bigcup_{i \in \N}U_i$ with $0<\vert U_i\vert \leq \delta$ for all $i$, we say that $\{U_i\}_{i \in \N}$ is a
$\delta$-covering of $E$.
\end{definition}

\begin{definition}
The space of dimension functions is defined as

\begin{equation*}
     H:= \left\{
      \begin{array}{ll}
h:\R_{\geq 0}\to\R_{\geq 0}     & \mathrm{ :\ } h(t)>0 \text{ if } t>0,
h(0)=0, \\
       & \mathrm{   } \text{ increasing and right-continuous}
       \end{array}
     \right\}.
\end{equation*}
\end{definition}

This set is partially ordered, considering the order defined by $$h_2 \prec h_1 \text{ if } \lim_{x \to 0^+} \frac{h_1(x)}{h_2(x)} =0.$$

\begin{definition}
The outer Hausdorff measure associated with $h$ is
$$\mathcal{H}^h(E):=\lim_{\delta\to 0} \inf_{} \left\{ \sum_{i}h(|U_i|) : \{U_i\}_{i} \text{ a } \delta
\text{-covering of } E \right\}.$$
\end{definition}
For any $h \in H$,  $\mathcal{H}^h$ is a Borel measure.
This definition generalises the outer $\alpha$-dimensional Hausdorff measure, which is the particular case  $h(x):=x^{\alpha}$. The relation of order says that $x^s \prec x^t$ if and only if $s<t$.

Dimension functions $h$ play a role analogue to the functions $h$ in the discrete case. Here it makes sense to define $h_2 \prec h_1$  if $\lim_{N \to \infty}\frac{h_2(N)}{h_1(N)}=0$, so that a set with $h_1(N)$ elements is larger than a set with $h_2(N)$ elements if $N$ is large enough.
In the continuous case, we have that if $E$ is a set, $h_2 \prec h_1$ and $\mathcal{H}^{h_1}(E)>0$, then $E$ does not have $\sigma$-finite $\mathcal{H}^{h_2}$ measure.
In particular, if $E$ is a set, and $h_2 \prec h_1$; then $\mathcal{H}^{h_1}(E) \leq \mathcal{H}^{h_2}(E)$.
So, intuitively we have that a set $E_1$ with $\mathcal{H}^{h_1}(E_1)>0$ is larger than a set $E_2$ with $\mathcal{H}^{h_2}(E_2)>0$.
Thus in both cases the function $h$ indicates the size of the set.

Keleti \cite{Kel} proved that there exists a compact set $E\subset\R$ of Hausdorff dimension $1$ that does not contain arithmetic progressions of length $3$. It is possible to construct such a set based on the Behrend example mentioned above, but Keleti did it directly using the existence of infinite scales in $\R$, which is the main difference with the discrete context. On the other hand, it is well known and easy to see that if $E \subset \R$ has positive Lebesgue measure, then $E$ contains arithmetic progressions of any finite length.

We can reinterpret these results in terms of dimension functions: there exists a compact set $E$ without arithmetic progressions such that $\mathcal{H}^{x^{1-\delta}}(E)>0$ for all $\delta>0$, but if $\mathcal{H}^x(E)>0$, then $E$ contains arbitrarily long arithmetic progressions.  It is natural then to investigate what happens to more general $h$ satisfying $h\prec x$.

Keleti \cite{Kel2} also constructed a compact set $E \subseteq \R$ with full Hausdorff dimension, which does not contain points  $x_1 < x_2 \leq x_3 < x_4$ such that $x_2-x_1-x_4+x_3=0$. This is an example of a linear pattern, which we now define formally:

\begin{definition}
 Given $E \subseteq \R^d$ we say that $\psi: \R^{dk} \to \R^n$ is a pattern in $E$, if there exist distinct $\vv{x_1}, \cdots, \vv{x_k} \in E$ such that $\psi(\vv{x_1}, \cdots , \vv{x_k})=\vv{0}$.

 In the particular case that $\psi$ is a linear function, we say that it is a linear pattern.
\end{definition}

Moreover, in \cite{Kel}, Keleti proved that given countable many triplets, there exists a compact set $E\subset\R$ of Hausdorff dimension $1$ that does not contain any rescaled and translated copy of any of the triplets. Maga \cite{Mag} extended Keleti's constructions to the plane. A particular case had been proved by Falconer \cite{Fal}.

Assuming additional hypotheses on the decay of the Fourier transform, in \cite[Corollary 1.7]{CLP} Chan, Laba and Pramanik ensure the existence of vertices of equilateral triangles in the plane.
For recent progress in dimension $4$ and higher, whithout assuming Fourier bounds, see Iosevich's and Liu's work \cite{Ios}.

It follows from a theorem of Máthé \cite[Theorem 6.1]{Mat} that there exists a compact set in $\R^d$ avoiding any countable family of linear patterns.
This generalises the results of Keleti and Maga mentioned above. In all these works all the authors considered only Hausdorff dimension.
In our main result, we obtain finer information by consider general dimension functions.


\begin{mainthm}\label{supergeneral}
 Let $h$ be a dimension function  with $h \prec x^d$, and let $(\psi_k)_{k \in \N}$ be a sequence of non-zero linear functions with  $\psi_k: (\R^d)^{m_k} \to \R$ and $m_k \geq 2$.
 Then, there exists a compact set $E \subseteq \R^d$ such that $\mathcal{H}^h(E)>0$, and $\psi_k(\vv{x_1}, \cdots ,\vv{x_{m_k}})\neq 0$ for all $k \in \N$ and all distinct vectors $\vv{x_1}, \cdots, \vv{x_{m_k}} \in E$.
\end{mainthm}

 In particular, if we choose $h(x):=-\log(x)x^d$, we obtain a set of Hausdorff dimension $d$ with the same properties.

\begin{remark}
A theorem of Besicovitch \cite{besicovitch} implies that if $E \subset \R^d$ has positive $\mathcal{H}^h$ measure for all $h \prec x^d$, then it has positive Lebesgue measure. Hence, since a set of positive Lebesgue measure contains every finite pattern, in Theorem \ref{supergeneral} it is not possible to have a set $E$ that works for every $h \prec x^d$.
\end{remark}


We remark that Máthé \cite{Mat}, and Fraser and Pramanik \cite{FP} studied similar problems for non-linear patterns, under certain conditions, but the sets they construct are not of full dimension, and in some particular cases the dimension obtained is optimal (that is to say, in some cases, there is no set of full dimension without the given non-linear patterns).
Since we want to study large sets for an arbitrary dimension function $h$ with $h \prec x^d$, we focus on the case of linear patterns.

In the opposite direction, in \cite{UMAY} Molter and the author proved that for any dimension function $h$ there exists a perfect set in the real line with $h$-Hausdorff measure zero that contains every polynomial pattern. In particular, there exists a perfect set in the real line with Hausdorff dimension zero that contains every polynomial pattern.

To prove our main theorem we still use Keleti's and Maga's idea of defining the cubes to kill the patterns at later stages of the construction, but the details are different. For example, we do not need to have separation between the cubes of the same level, but we need a uniform bound for the amount of offspring of each cube. We also have to modify the location of the cubes to fit any linear pattern.

See Section \ref{demoteo} for the proof of the main result (Theorem \ref{supergeneral}), and Section \ref{seccionaplicaciones} for several concrete applications.




\textbf{Acknowledgement.} We thank the referee for several useful comments.


\section{The proof of Theorem \ref{supergeneral}}\label{demoteo}

Before proving the theorem, we review some preliminaries.

A mass distribution  $\mu$ on $E$ is an outer measure with $0<\mu(E)<+\infty$.

The following proposition is well-known but we give the proof since we were not able to find a reference.

\begin{proposition} [Generalized mass distribution principle]\label{MasaGen}
Let $\mu$ a mass distribution on $E$ and let $h$ be a dimension function such that there exist $c>0$, $\varepsilon>0$ satisfying that
$$\mu(U)\leq c h( \vert U\vert) \ \ \forall U \text{ with } \ \vert U\vert\ \leq
\varepsilon.$$
Then $$0 < \frac{ \mu(E)}{c} \leq\mathcal{H}^h(E).$$
\end{proposition}

\begin{proof}
Let $\delta \in (0, \varepsilon)$. If $\{U_i\}_{i\in \N}$ is a  $\delta$-covering of $E$, then
$$0<\mu(E)\leq\mu\left( \bigcup_{i\in \N}U_i\right)\leq \sum_{i \in \N}\mu(U_i)\leq c
\sum_{i \in \N}h(|U_i|).$$
Taking infimum on the $\delta$-coverings, we have $$0<\mu(E)\leq
c\mathcal{H}^h_{\delta}(E).$$
And then we can take the limit $\delta \to 0$ to obtain $0<\mu(E)\leq c\mathcal{H}^h (E)$,
so
$$0<\frac{\mu(E)}{c}\leq \mathcal{H}^h (E).$$
\end{proof}

If we have a set $E:=\bigcap_{k \in \N_0}E_k$, where $E_k$ is a nested sequence of finite unions of closed non-overlapping cubes, we say that a cube $I$ of the construction of the set $E$ is a cube of level $k$ if $I$ is one of the cubes making up $E_k$. We also say that a cube $J$ is an ancestor of $I$ if $J, I$ are cubes of levels $j,k$ respectively, with $j<k$ and $I \subseteq J$.

\begin{lemma}\label{CoroMasa}
 Let $E:=\bigcap_{k\in\N} E_k \subseteq \R^d$ where $E_k$ is a finite union of non-overlaping cubes of the same size, and each cube of $E_{k+1}$ is contained in a cube of $E_k$.
 Let $\mu$ a mass distribution on $E$ and let $h$ be a dimension function such that there exists $k_0 \in \N$ satisfying $\mu(I_k)\leq c_1 h(|I_k|)$ for all $I_k$ cube of level $k$, for all $k \geq k_0$.
 Suppose each cube of $E_k$ contains at  most $c_2$ cubes of the level $k+1$.

 Then there exists a constant $c_3$ depending on $c_1$, $c_2$ and $d$ such that $$0<\frac{\mu(E)}{c_3}\leq \mathcal{H}^h(E).$$
\end{lemma}

\begin{proof}
 We will use Proposition \ref{MasaGen}.

 Let $\varepsilon:=|I_{k_0}|\sqrt{d}$ be the diameter of any cube in the level $k_0$. Let us write $\delta_k$ for the side-length of any cube of level $k$.
 If $U$ is a set with $|U|\in(0,\varepsilon)$, then there exists $k \geq k_0$ such that $\sqrt{d}\delta_{k+1}\leq |U| < \sqrt{d}\delta_k$.
 There exists a cube $C$ with side-length $2 \sqrt{d}\delta_k$ such that $U \subseteq C$. Since $C$ intesects at most $(2 \sqrt{d} +3)^d$ cubes  of level $k$, then by hypotesis $C$ intersects at most $c_2 (2 \sqrt{d} +3)^d$ cubes $I_j$ of level $k+1$.
 Then there exists a constant $c_3$ depending on $c_1$, $c_2$ and $d$ such that $$\mu(U)\leq \sum^{c_2 (2 \sqrt{d} +3)^d}_{j=1}\mu(I_j)\leq c_2 (2 \sqrt{d} +3)^d c_1 h(\sqrt{d} \delta_{k+1})\leq c_3 h(|U|).$$


 Using this fact and Proposition \ref{MasaGen}, the result follows.
 \end{proof}


Now, we are able to prove the main theorem.

\begin{proof}[Proof of Theorem \ref{supergeneral}]

We can assume that each function appears infinitely many times in the sequence $\{ \psi_k \}_{k \in \N}$.

We will construct a set $E:=\bigcap_{k \in \N_0}E_k \subseteq [1,2]^d$, where $E_k$ is a nested sequence of finite unions of non-overlapping closed cubes, so the set $E$ will be compact.

As $\psi_i$ is a non-zero linear function, we can define $$c_i:=\max_{\vv{y}\in \left[ -\frac{1}{2}, \frac{1}{2} \right]^{m_i d}} | \psi_i(\vv{y}) | >0.$$

The function $\psi_i$ has the form $$\psi_i(x_{1,1}, \cdots,x_{1,d}, \cdots , x_{m_i,1}, \cdots , x_{m_i,d}):=b_{i,1,1} x_{1,1}+ \cdots + b_{i,m_i,d} x_{m_i,d}.$$

Since permuting the sets $(x_{k,1},\ldots,x_{k,d})$ with $1 \leq k \leq m_i$, and multiplying $\psi_i$ by a non-zero constant, do not affect the statement, we can assume without loss of generality that there exists $j_i \in \{1, \cdots ,d\}$ such that $b_{i,m_i,j_i}=1$.

Let $\lambda_{i,\ell,v}:=\frac{1}{|b_{i,\ell,v}|}$ if $b_{i,\ell,v} \neq 0$, and $\lambda_{i,\ell,v}:=1$ if not. Thus, we have that $$\lambda_{i,\ell,v} b_{i,\ell,v}= \text{sg} (b_{i,\ell,v}) \text{ for all }i, 1 \leq \ell \leq m_i , 1 \leq v \leq d,$$  where $\text{sg}$ is the sign function.

For each $i$ and each $1 \leq \ell \leq m_i-1$, we define  the function $$\phi^\ell_i(x_1, \cdots , x_d):=(\lambda_{i,\ell,1}x_1 , \cdots , \lambda_{i,\ell,d} x_d),$$ and for $\ell=m_i$ we set $$\phi^{m_i}_i(x_1, \cdots , x_d):=(\lambda_{i,m_i,1}x_1 , \cdots , \lambda_{i,m_i,d} x_d)+\frac{1}{2} e_{j_i},$$ where $e_j=(v_1, \cdots , v_d)$ with $v_j=1$ and $v_k=0$ for all $k\neq j$.

As a consecuence of these definitions, we have
\begin{align*}
 \psi_i  (\phi^1_i (\Z^d), \cdots ,\phi^{m_i}_i (\Z^d)) &= \text{sg}(b_{i,1,1})\Z + \cdots + \text{sg}(b_{i,m_i,d})\Z +\frac{1}{2} \\
 &=\Z +\frac{1}{2}.
\end{align*}

Therefore, we have
\begin{equation}\label{cotainfnorma}
| \psi_i  (\phi^1_i (\vv{z_1}), \cdots ,\phi^{m_i}_i (\vv{z_{m_i}})) | \geq \frac{1}{2}
\end{equation}
for all $\vv{z_1}, \cdots, \vv{z_{m_i}} \in \Z^d$, for all $i$.

We define $\beta_i$ such that $\beta_i \geq m_i$ and $$\frac{\beta_i}{2}\geq \max \{ \lambda_{i,\ell,v} :  1\leq \ell \leq m_i, \ 1\leq v \leq d \}  2c_i \sqrt{d}+\frac{\sqrt{d}}{2}.$$

Let $(M_i)_{i \in \N} \subseteq \N_{\geq 2}$ be a strictly increasing sequence such that for all $i$:
\begin{itemize}
 \item $M_{i+1} \geq M_i +2$
 \item $\frac{h(\sqrt{d} 2^{-k} \prod_{j : M_j \leq k} \beta_j^{-1})}{\left( \sqrt{d}2^{-k} \prod_{j : M_j \leq k} \beta_j^{-1} \right)^d} \geq 2^{id} \beta_1^d \cdots \beta_i^d$ for all $k \geq M_i$
\end{itemize}

The last condition holds if $M_i$ is large enough by the assumption $h\prec x^d$ and $\sqrt{d}2^{-k} \prod_{j : M_j \leq k} \beta_j^{-1} \leq \sqrt{d}2^{-k}$ tends to $0$ independently of $\beta_1, \cdots, \beta_i$.

We will construct $E$ avoiding the given patterns in the levels $\{ M_i \}_i$. Let $E_0:=[1,2]^d$.
We will construct $E_k$ as a union of  $N_k:= 2^{d \left(k- \# \{j\leq k: \ j\in \bigcup_i\{M_i\}  \} \right)}$ cubes with side-length $\delta_k:=2^{-k} \prod_{i : \ M_i\leq k}\beta_i^{-1}$.

For each $i$, let $\Gamma_{m_i} :=\{(J_k^1, \cdots, J_k^{m_i}) \}_{k \in \N}$ be the set of all $m_i$-tuples of different cubes of the same level of construction of $E$ (we ignore levels with fewer than $m_i$ cubes), in every possible order, and write $\Gamma=\bigcup_i \Gamma_{m_i}$. At this point, the notation $J_k^j$ should be understood as a label for the actual cube which is yet to be defined (note, however, that the number of cubes and their sizes are already fixed). In the construction below, we will inductively (in $k$) define the positions of the cube corresponding to each label.

Let $( U_j )_{j \in \N}$ be a sequence where:
\begin{itemize}
\item each element of $\Gamma$ appears infinitely many times
\item $U_j \in \Gamma_{m_j}$
\item for all $i$  and for all $U \in \Gamma_{m_i}$ there exists $k \in \N$ such that $\psi_k=\psi_i$, $U=U_k$ and every cube of $U_k$ is of level $< k-1$.
\end{itemize}

For this, it is enough that for each $i$, if we consider the subsequence $(\psi_{j_n})_{n \in \N}$ of all terms which are equal to $\psi_i$, we require that each $U_{j_n}$ is an element of $\Gamma_{m_i}$ and, additionally, each element of $\Gamma_{m_i}$ appears infinitely often in the subsequence $(U_{j_n})_{n \in \N}$.

Once $E_{k-1}$ is given, the construction of $E_k$ depends on whether $k$ belongs to $(M_i)_{i \in \N}$:

\begin{enumerate}[(a)]
 \item If $k \notin \bigcup_{i \in \N} \{M_i\}$, we will split each cube of level $k-1$ into $2^d$ closed cubes of the same size.
 \item If $k=M_i$ for some $i$, we will do different things, depending on whether they have some ancestor among the cubes of $U_i$: $J_i^1, \cdots , J_i^{m_i}$.

 For each cube $I$ of level $k-1$ which is not contained in any of the cubes in the tuple $U_i$, we will take any cube $I' \subseteq I$ with side-length $\delta_k$.

 For each cube $I$ of level $k-1$ which is contained in some cube $J^{\ell}_i$ of $U_i$, we will take a cube $I' \subseteq I$ of the form \begin{equation}\label{formacubo}
\delta_{M_i} \left( 4c_i \phi_i^\ell(\vv{z_\ell})+ \left[ -\frac{1}{2}, \frac{1}{2} \right]^d \right) \text{ with } \vv{z_\ell} \in \Z^d.
\end{equation}

\end{enumerate}

We let $E_k$ be the union of all the cubes $I'$.

Let us see that the cubes can indeed be taken in this way:

In case (a) this is clear, because $\delta_k=\frac{\delta_{k-1}}{2}$.

In case (b), let $I$ be a cube of $E_{k-1}$ that is a contained in a cube $J^\ell_k$ of $U_k$.
As $I$ is a cube with side-length $\delta_{k-1}=\beta_i \delta_k$, we have that $\frac{1}{\delta_{k}}I$ is a closed cube of side-length $\beta_i$, so it contains a closed ball of radius $\frac{\beta_i}{2}$, whose center we will denote by $\vv{x}$.

By the definition of $\phi_i^\ell$, there exists $\vv{z} \in \Z^d$ such that $$\dist(\vv{x}, 4c_i \phi^{\ell}_i ( \vv{z})) \leq \max \{ \lambda_{i,\ell,1}, \cdots , \lambda_{i,\ell,d}\} 2c_i  \sqrt{d}.$$ Using this and by the assumptions on $\beta_i$, we have:
\begin{align*}
 4c_i \phi_i^\ell (\vv{z})+ \left[ -\frac{1}{2}, \frac{1}{2} \right]^d &\subseteq B\left[4c_i \phi_i^\ell( \vv{z}),\frac{\sqrt{d}}{2} \right]\\
 &\subseteq B\left[\vv{x}, \max \{ \lambda_{i,\ell,1}, \cdots , \lambda_{i,\ell,d}\} 2c_i  \sqrt{d}+\frac{\sqrt{d}}{2} \right]\\
 &\subseteq B\left[ \vv{x}, \frac{\beta_i}{2} \right] \subseteq \frac{1}{\delta_k}I.
\end{align*}

\textbf{Claim:} If $n\in\N$ and $\vv{x_1}, \cdots , \vv{x_{m_n}} \in E \subset \R^d$ are distinct, then $\psi_n(\vv{x_1}, \cdots , \vv{x_{m_n}}) \neq 0$.

We will prove the claim by contradiction. Suppose that $\psi_n(\vv{x_1}, \cdots , \vv{x_{m_n}}) = 0$.
Since $\vv{x_1}, \cdots , \vv{x_{m_n}} \in E$ are distinct, by definition of the sequence $(U_k)_{k \in \N}$, there exists $i \in \N$ such that $\psi_n=\psi_i$, $\vv{x_1} \in J^1_i$, $\cdots$, $\vv{x_{m_i}} \in J_i^{m_i}$, and every $J_i^\ell$ is of the same level $< i-1$.

Considering the level $M_i\geq i$, we see from \eqref{formacubo} that $\vv{x_\ell}=\delta_{M_i}\left(4c_i \phi_i^\ell (\vv{z_\ell})+ \vv{\Delta_\ell} \right)$ with $\vv{z_\ell} \in \Z^d$ and $\vv{\Delta_\ell}\in \left[ -\frac{1}{2}, \frac{1}{2} \right]^d$ for all $1\leq \ell \leq m_i$.
Hence, by linearity of $\psi_i$ we get $$4c_i \psi_i ( \phi_i^1 ( \vv{z_1}), \cdots ,  \phi_i^{m_i} (\vv{z_{m_i}}))+ \psi_i (\vv{\Delta_1}, \cdots , \vv{\Delta_m})=0$$ with $ \vv{z_1}, \cdots ,  \vv{z_{m_i}} \in \Z^d$.
Hence, we have by \eqref{cotainfnorma} $$2 c_i \leq 4c_i | \psi_i (\phi^i_1 (\vv{z_1}), \cdots , \phi_m^i ( \vv{z_m})) |= | \psi_i (\vv{\Delta_1}, \cdots , \vv{\Delta_m}) | \leq c_i$$  which is a contradiction.

\textbf{Claim:} $\mathcal{H}^h(E)>0$.

Let $\mu$ be the uniform mass distribution, i.e.: $\mu(I_k)=\frac{1}{\# \text{cubes of level } k}$ for each $I_k$ cube of level $k$.
It is enough to prove that if $I$ is a cube of level $k$ with $k$ large enough then $$h(|I|)\geq \frac{1}{\# \text{cubes of level } k}.$$ The claim will then follow from Lemma \ref{CoroMasa} and the fact that each cube of $E_k$ has at most $2^d$ offspring of level $k+1$.

The side-length of $I$ is $\delta_k=2^{-k} \prod_{j: \ M_j\leq k}\beta_j^{-1}$. If $k$ is large enough there exists $j$ such that $M_j \leq k < M_{j+1}$.
By definition of $(M_i)_{i \in \N}$, we have
\begin{align*}
 h(|I|)&=h \left(\sqrt{d} 2^{-k} \prod_{i: \ M_i\leq k}\beta_i^{-1} \right) \\
 &\geq \left( \sqrt{d} 2^{-k} \prod_{i: \ M_i\leq k}\beta_i^{-1} \right)^d 2^{jd} \beta_1^d \cdots \beta_j^d\\
 &\geq \frac{1}{2^{d(k-j)}}=\frac{1}{\# \text{cubes of level }k}.
\end{align*}

\end{proof}


\section{Applications}\label{seccionaplicaciones}

In this section we will present some applications of Theorem \ref{supergeneral}.

\begin{corollary}\label{quotients}
Given  a dimension function $h$ with $h \prec x$ and a countable set $A \subseteq \R_{\neq 1}$, there exists a compact set $E \subseteq [1,2]$ such that $\mathcal{H}^h(E)>0$ and the set of quotients of $E$ given by $\frac{E}{E}:=\{\frac{y}{x} \ : \ x,y \in E \}$ does not contain any element of $A$.
\end{corollary}

\begin{proof}
We choose $d=1$, $m_a=2$ for all $k$, $\psi_a(x,y):=ax-y$ for all $a \in A$, and apply Theorem \ref{supergeneral}, the corollary follows.
\end{proof}

\begin{corollary}
Given  a dimension function $h$ with $h \prec x$ and a countable set $\tilde{A} \subseteq \R_{\neq 0}$, there exists a compact set $\tilde{E} \subseteq [0,\log(2)]$ such that $\mathcal{H}^h(\tilde{E})>0$ and the set of differences of $\tilde{E}$ given by $\tilde{E}-\tilde{E}:=\{y-x \ : \ x,y \in \tilde{E} \}$ does not contain any element of $\tilde{A}$.
\end{corollary}

\begin{proof}
We define $A:=e^{\tilde{A}}:=\{ e^{\tilde{a}}: \ \tilde{a} \in \tilde{A} \} \subseteq (0, +\infty) \setminus \{ 1 \}$.
By Corollary \ref{quotients} we have a compact set $E \subseteq [1,2]$ such that $\mathcal{H}^h(E)>0$ and $\frac{y}{x}\neq a$ for all distinct $x,y \in E$ and all $a \in A$.
We choose $\tilde{E}:=\log (E)$. We have $\mathcal{H}^h(\tilde{E})>0$, because $\log|_{[1,2]} : [1,2] \to [0, \log(2)]$ is a bilipschitz function. For every $\tilde{a} \in \tilde{A}$, and distinct $\tilde{x}, \tilde{y} \in \tilde{E}$, we have $\tilde{a}=\log(a)$ with $a \in A$, $\tilde{x}=\log(x)$ and $\tilde{y}=\log(y)$ where $x, y \in E$ are distinct, so $$\tilde{y}-\tilde{x}=\log(\frac{y}{x})\neq \log(a)=\tilde{a}.$$
\end{proof}

In particular, if $\tilde{A}$ is a countable and dense set, we obtain a set $\tilde{E}$ of positive $\mathcal{H}^h$-measure whose set of differences has empty interior.
This contrasts with Steinhaus' Theorem, asserting that the difference set of a set of positive Lebesgue measure contains an interval.

\begin{corollary}\label{planos}
Given a dimension function $h$ with $h \prec x$ and given a countable set of planes  $\{ \pi_k \}_k$ in $\R^3$ containing the origin, there exists a compact set $E \subseteq \R$ with $\mathcal{H}^h (E)>0$ such that $$(x,y,z) \notin \pi_k \ \forall k \text{ for all distinct } x,y,z \in E.$$
\end{corollary}

\begin{proof}
Each of those planes $\pi_k$ is given by an equation $\psi_k(x,y,z):= a_k x+ b_k y + c_k z =0$.
Taking $d=1$, $m_k=3$ for all $k$, and $\psi_k$ as above, and applying Theorem \ref{supergeneral} the result follows.
\end{proof}

\begin{corollary}\label{proporciones}
Given a dimension function $h$ with $h \prec x$ and a countable set $A \subseteq (1,+\infty)$, there exists a compact set $E \subseteq \R$ with $\mathcal{H}^h (E)>0$ such that $$\frac{z-x}{z-y}\notin A \quad \forall x<y<z \text{ in } E.$$
\end{corollary}

\begin{proof}
If we choose $\pi_k: \   x-\alpha_k y+(\alpha_k -1)z=0$ in Corollary \ref{planos}, the result follows.
\end{proof}

In particular, choosing $A=\{2\}$, there exists a compact set $E \subseteq \R$  with $\mathcal{H}^h (E)>0$ that does not contain any arithmetic progression of length $3$. Note that choosing e.g. $h(x):=-\log(x)x$, we recover the result of Keleti (\cite{Kel}) mentioned in the introduction.

\begin{corollary}\label{masvariables}
It is an equivalent result if we consider $\psi_k: \R^{m_kd} \to \R^{N_k}$ in Theorem \ref{supergeneral}.
\end{corollary}
\begin{proof}
It is clear that this is more general than Theorem \ref{supergeneral}. And Theorem \ref{supergeneral} implies this statement, since given $\psi_k: \R^{m_kd} \to \R^{N_k}$ we can separate it into $N_k$ linear functions, discard the zero functions, and apply the theorem.
\end{proof}

\begin{corollary}
Let $d \in \N$ and let $h$ a dimension function such that $h \prec x^d$.
Then there exists a compact set $E \subseteq \R^d$ such that $\mathcal{H}^h(E)>0$, and $E$ does not contain the vertices of any parallelogram.
\end{corollary}
\begin{proof}
This follows from Theorem \ref{supergeneral} and Corollary \ref{masvariables}, taking $$\psi:\R^{4d} \to \R^d, \ \ \ \psi(\vv{x_1},\vv{x_2},\vv{x_3},\vv{x_4}):=\vv{x_1}-\vv{x_2}+\vv{x_3}-\vv{x_4}.$$
\end{proof}
The previous corollary is an improvement over Maga's result \cite[Theorem 2.3]{Mag}.

\begin{corollary}\label{trapecios}
Let $d \in \N$ and let $h$ a dimension function such that $h \prec x^d$.
Let $(\alpha_n)_{n \in \N} \subseteq \R_{\neq 0}$. We get a compact set $E \subseteq \R^d$ such that $\mathcal{H}^h(E)>0$, and for all $n \in \N$, $E$ does not contain the vertices of any trapezoid with the lengths of the parallel sides in proportion $\alpha_n$.
\end{corollary}

\begin{proof}
Taking $\psi_n: \R^{4d} \to \R^d$ given by $\psi_n(\vv{x_1}, \vv{x_2}, \vv{x_3}, \vv{x_4}):= \vv{x_1}- \vv{x_2}- \alpha_n (\vv{x_3}-\vv{x_4})$, the result follows by applying Theorem \ref{supergeneral} and Corollary \ref{masvariables}.
\end{proof}

We have the following complex version:
\begin{corollary}\label{complejogeneral}
Let $h$  be a dimension function with $h \prec x^{2s}$ (with $s \in \N$), $m \geq 2$, and consider a sequence of $\R$-linear functions  $(\psi_k)_{k \in \N}$ such that $\psi_k: \C^{ms} \to \C$.

Then there exists a compact set $E \subseteq \C^s$ such that $\mathcal{H}^h(E)>0$ and $\psi_k(\vv{x_1}, \cdots ,\vv{x_m})\neq 0$ for all distinct $\vv{x_1}, \cdots, \vv{x_m} \in E$.
\end{corollary}
\begin{proof}
We take $d=2s$ and identify $\C$ with $\R^2$ in Theorem \ref{supergeneral} and Corollary \ref{masvariables}.
\end{proof}

\begin{corollary}\label{teo contables ternas}
Let $h$ be a dimension function with $h \prec x^2$, and let $(P_n)_{n\in \N}=(x_n,y_n,z_n)_{n\in \N}$ a sequence of triplets of different complex numbers.
Then there exists a compact set $E \subseteq \C$, with $\mathcal{H}^h(E)>0$, that does not contain a similar copy of any of the given triplets
\end{corollary}
 \begin{proof}
Take $m=3$ and for each $n \in \N$ define $\alpha_n:=\frac{z_n - x_n}{z_n - y_n}$, $\psi_n(x,y,z)=(\alpha_n -1)z-\alpha_n y+x$, and apply Corollary \ref{complejogeneral}.
 \end{proof}

In particular, taking $h(x):=-x^2 \log(x)$, we recover the results of Maga \cite[Theorem 2.8]{Mag}, and Falconer \cite{Fal} mentioned in the introduction.



\end{document}